\documentclass[alphabetic, 12pt]{amsart}

\usepackage{enumerate, longtable}
\usepackage{amsmath, amscd, amsfonts, amsthm, amssymb, latexsym, comment, stmaryrd, graphicx, dsfont}
\usepackage{xcolor}
\usepackage[all]{xy}
\usepackage{fullpage}
\usepackage{tikz}
\usepackage[
        colorlinks,
        linkcolor=red,  citecolor=blue,
        backref
]{hyperref}

\usepackage{amstext} 
\usepackage{array}
\usepackage{amsrefs}
\usepackage{cleveref}

\author{Alexei Entin}
\address{Raymond and Beverly Sackler School of Mathematical Sciences, Tel Aviv University, Tel Aviv 69978, Israel}
\email{aentin@tauex.tau.ac.il}
\author{Sean Landsberg}
\address{Raymond and Beverly Sackler School of Mathematical Sciences, Tel Aviv University, Tel Aviv 69978, Israel}
\email{seanl@mail.tau.ac.il}


\title{The Least Common Multiple of Polynomial Values over Function Fields}

\date{}

\newcommand{\LL}{L_f(n) }
\newcommand{\PP}{P_f(n) }
\newcommand{\F}{{\mathbb{F}}}
\newcommand{\Z}{{\mathbb{Z}}}

\newcommand{\A}{\mathbb{F}_q[T]}
\newcommand{\AX}{\mathbb{F}_q[T][X]}
\newcommand{\val}[1]{v_P(#1) }
\newcommand{\an}{\alpha_P(n) }
\newcommand{\bn}{\beta_P(n) }
\newcommand{\p}[1]{\rho_f(P^{#1})}
\newcommand{\forb}[1]{\left( \frac{E / \F_q(T)}{#1} \right)}

\newtheorem{claim}{Claim}[section]

\newtheorem{conj}[claim]{Conjecture} 
\newtheorem{thrm}[claim]{Theorem}
\newtheorem{prop}[claim]{Proposition}

\newtheorem{lem}[claim]{Lemma}

\theoremstyle{definition}

\newtheorem{defn}[claim]{Definition} 
\newtheorem{rem}[claim]{Remark}
\newtheorem{ex}[claim]{Example}

\begin{document}
\maketitle
\begin{abstract}
    Cilleruelo conjectured that for an irreducible polynomial $f \in \mathbb{Z}[X]$ of degree $d \geq 2$ one has $$\log\left[\mathrm{lcm}(f(1),f(2),\ldots f(N))\right]\sim(d-1)N\log N$$ as $N \to \infty$. He proved it in the case $d=2$ but it remains open for every polynomial with $d>2$.
    
    We  investigate the function field analogue of the problem by 
    considering polynomials over the ring $\F_q[T]$. We state an analog 
    of Cilleruelo's conjecture in this setting: denoting by $$\LL 
    := \mathrm{lcm} \left(f\left(Q\right)\ : \ Q \in \F_q[T]\mbox{ 
    monic},\, \deg Q = n\right)$$ we conjecture that 
    \begin{equation}\label{eq:conjffabs}\deg\LL \sim c_f \left(d-1\right) 
    nq^n,\ n \to \infty\end{equation} ($c_f$ is an explicit constant 
    dependent only on $f$, typically $c_f=1$).
    We give both upper and lower bounds for $L_f(n)$ and show that 
    the asymptotic (\ref{eq:conjffabs}) holds for a class of ``special" 
    polynomials, initially considered by Leumi in this context, which 
    includes all quadratic polynomials and many other examples as 
    well. We fully classify these special polynomials. We also show 
    that $\deg L_f(n) \sim  \deg\mathrm{rad}\left(\LL\right)$ (in other words the corresponding LCM is close to being squarefree), which is not known over $\Z$.
    \\ \\

\end{abstract}

\numberwithin{equation}{section}

\section{Introduction}
While studying the distribution of prime numbers, Chebychev estimated the least common multiple of the first $N$ integers. This was an important step towards the prime number theorem. In fact the condition $\textrm{log lcm}\left(1 ,\ldots, N\right) \sim N$ is equivalent to the prime number theorem. This problem later inspired a more general problem of studying the least common multiple of polynomial sequences. For a linear polynomial $f(X) = k X + h,\ k,h\in\Z,\ k>0, h+k > 0$ it was observed by Bateman and Kalb \cite{bateman2002limit}, that
$\textrm{log lcm}\left(f\left(1\right) ,\ldots, f\left(N\right)\right) \sim c_f N$ as $N\to\infty$, where $c_f = \frac{k}{\varphi(k)} \sum_{1 \leq m \leq k,(m,k) = 1} \frac{1}{m}$, which is a consequence of the Prime Number Theorem for arithmetic progressions. 

Cilleruelo conjectured \cite{cilleruelo2011least} that for any irreducible polynomial $f \in \mathbb{Z}[X]$ of degree $d\ge 2$ the following estimate holds:
\begin{equation}\label{eq:conj_intro}\textrm{log lcm}\left(f\left(1\right) ,..., f\left(N\right)\right) \sim \left(d-1\right)N \log N\end{equation} as $f$ is fixed and $N\to\infty$.
\\ \\
{\bf Convention.} Throughout the rest of the paper in all asymptotic notation the polynomial $f$, and the parameter $q$ to appear later, are assumed fixed, while the parameter $N$ (or $n$) is taken to infinity. If any parameters other than $f,q,N,n$ appear in the notation, the implied constant or rate of convergence are uniform in these parameters. 
\\ \\
Cilleruelo proved (\ref{eq:conj_intro}) for quadratic polynomials and there are no known examples of polynomials of degree $d > 2$ for which the conjecture is known to hold. Cilleruelo's argument also shows the predicted upper bound, i.e. if $f \in \mathbb{Z}[X]$ is an irreducible polynomial of degree $d \geq 2$, then
$$\log\mathrm{lcm}\,\left(f(1),\ldots,f(N)\right)\lesssim (d-1)N\log N.$$

Maynard and Rudnick \cite{maynard2019lower} provided a lower bound of the correct order of magnitude:
\begin{equation}\label{eq:sah_lower_bound}\textrm{log lcm}\left(f(1),\ldots, f(N)\right) \gtrsim \frac{1}{d} N \log N.\end{equation}
Sah \cite{sah2020improved} improved upon this lower bound while also providing a lower bound for the radical of the least common multiple:
$$\textrm{log lcm}\left(f\left(1\right) ,\ldots, f\left(N\right)\right) \gtrsim N \log N,$$

\begin{equation}\label{eq:sah_lower_bound_rad}\textrm{log rad}\left[ \textrm{lcm}\left(f\left(1\right) ,..., f\left(N\right)\right)\right] \gtrsim \frac{2}{d} N \log N.\end{equation}

Rudnick and Zehavi \cite{rudnick2020cilleruelo} established an averaged form of (\ref{eq:conj_intro}) with $f$ varying in a suitable range. Leumi \cite{leumi2021lcm} studied a function field analogue of the problem. In the present work we expand upon and generalize the results and conjectures in \cite{leumi2021lcm}, as well as correct some erroneous statements and conjectures from the latter work. Despite some overlap, we have kept our exposition self-contained and independent of \cite{leumi2021lcm}.
 
\subsection{The function field analogue.}

Let $q=p^k$ be a prime power. For a polynomial $f \in \AX$ of degree $d\ge 1$ of the form 
\[f(X) = f_d X^d + f_{d-1} X^{d-1} + ... + f_0,\ f_i\in\F_q[T]\]
set
\begin{equation}\label{eq: def Lfn}\LL := \mathrm{lcm} \left(f\left(Q\right): \ Q \in M_n\right),\end{equation}
where $$M_n := \{Q \in \A \text{ monic},\ \deg Q = n\}.$$ Also denote
$$V_f := \left\{g \in \A:\ f\left(X + g\right) = f\left(X\right)\right\}.$$
The set $V_f$ is a finite-dimensional $\F_p$-linear subspace of $\A$ (see Lemma \ref{V_f: trivial when degree is not devisable by p} below). Denote $$c_f := \frac{1}{|V_f|}.$$ We now state a function field analog of Cilleruelo’s conjecture.

\begin{conj}
\label{the conjecture}
    Let $f \in \AX$ be a fixed irreducible polynomial with $\deg_X f = d \geq 2$. Then
    \[\deg \LL \sim c_f \left(d-1\right) nq^n,\ n \to \infty.\]
\end{conj}

\begin{rem}\label{rem: V_f char 0}
     The expression $(d-1)nq^n$ is directly analogous to $(d-1)N\log N$ appearing in (\ref{eq:conj_intro}). Over the integers (or generally in characteristic 0) if $f$ is not constant then $V_f$ is always trivial and no constant $c_f$ appears in (\ref{eq:conj_intro}). This is because if $0 \neq g \in V_f$ then $2g,...,d g$ are also in $V_f$ implying $f\left(0\right) = f\left(g\right) = ... = f\left(d g\right)$. Since $f\left(x\right) - f\left(0\right)$ has at most $d$ roots, we reach a contradiction. Even over $\F_q[T]$, the typical case that occurs for ``most'' polynomials is $c_f=1$ (i.e. $V_f$ is trivial). A heuristic justification of \Cref{the conjecture} will be given in \Cref{sec: heuristic}.
\end{rem}

In the present paper we prove that Conjecture \ref{the conjecture} gives the correct upper bound:

\begin{thrm}
\label{Upper bound}
    \[\deg L_f(n) \lesssim c_f(d-1)n q^n.\]
\end{thrm}
The proof of \Cref{Upper bound} will be given in \Cref{sec: upper bound}. We also give a lower bound of the correct order of magnitude (this bound is comparable to the bound in \cite[Theorem 1.3]{sah2020improved} under a mild assumption on $f$):

\begin{thrm}
    \label{lower bound}
    \begin{enumerate}
        \item[(i)]\[ \deg L_f(n) \gtrsim \frac{d-1}{d} n q^n.\]
        \item[(ii)] If $p \nmid d$ or $f_d \nmid f_{d-1}$ then
    \[\deg L_f(n) \gtrsim  n q^n.\]
    \end{enumerate}
\end{thrm}

The proof of \Cref{lower bound} will be given in \Cref{sec: lower bound}. We note that by Lemma \ref{V_f: trivial when degree is not devisable by p} below we always have $c_f\ge 1/d$ and $c_f\ge 1/(d-1)$ if $p\nmid d$ or $f_d\nmid f_{d-1}$, so this is consistent with Conjecture \ref{the conjecture}.

Regarding the radical of the LCM $$\ell_f(n):=\mathrm{rad}\ \mathrm{lcm}(f(Q):\ Q\in M_n)$$ we prove the following

\begin{thrm}\label{thm:rad}
\label{ell = L}
    \[\deg \ell_f(n) \sim \deg L_f(n).\]
\end{thrm}

The proof of \Cref{thm:rad} will be given in \Cref{sec: rad}. As a consequence, our lower bounds for $L_f(n)$ apply also to $\ell_f(n)$. The analogous statement over $\Z$ is not known and the best lower bound over $\Z$, namely (\ref{eq:sah_lower_bound_rad}), has a smaller constant (if $d>2$) than the best known lower bound for $L_f(N)$ given by (\ref{eq:sah_lower_bound}). The key ingredient in the proof which is unavailable over $\Z$ is the work of Poonen \cite{Poo03} on squarefree values of polynomials over $\F_q[T]$ (later generalized by Lando \cite{Lan15} and Carmon \cite{carmon2021square}). If one assumes the ABC conjecture, then the analogous statement $\log\ell_f(N)\sim\log L_f(N)$ can be proven over $\Z$ (see Remark \ref{rem: ABC} below).

For a class of polynomials $f\in\F_q[T][X]$ we call \emph{special} (first introduced by Leumi in \cite{leumi2021lcm}) it is possible to establish Conjecture \ref{the conjecture} in full. This class includes all quadratic polynomials, but also many polynomials of higher degree. We now define special polynomials over an arbitrary unique factorization domain (UFD). This definition was introduced in \cite{leumi2021lcm}.

\theoremstyle{definition}
\begin{defn}
    \label{special defention}
    A polynomial $f \in R[X]$ of degree $d=\deg f\ge 2$ is called \emph{special} in if the bivariate polynomial $f(X) - f(Y)$ factors into a product of linear terms in $R[X,Y]$:
    $$f(X)-f(Y)=\prod_{i=1}^d(a_iX+b_iY+c_i),\quad a_i,b_i,c_i\in R.$$
\end{defn}

\begin{ex} \label{ex: special}\begin{enumerate}
    \item[(i)] A quadratic polynomial is always special because
    $$AX^2+BX+C-(AY^2+BY+C)=(X-Y)(AX+AY+B).$$
    \item[(ii)] If $R=\F_p$ then $f=X^p$ is special because
    $$X^p-Y^p=\prod_{a\in\F_p}(X-aY).$$
\end{enumerate}

\end{ex}

For a special polynomial $f$ \Cref{the conjecture} can be established in full.

\begin{thrm}
    \label{special = conjecture}
    If $f\in\F_q[T][X]$ is irreducible and special then Conjecture \ref{the conjecture} holds for $f$.
\end{thrm}

The proof of \Cref{special = conjecture} will be given in section \Cref{sec: special asym}.

\begin{ex} The polynomial $f=X^p-T\in\F_p[T][X]$ is irreducible (since it is linear in $T$) and special (similarly to \Cref{ex: special}(ii)). Hence \Cref{the conjecture} holds for it.\end{ex}

We fully classify the set of special polynomials over an arbitrary UFD $R$.

\begin{thrm}
    \label{clasifaction of spesial polinomials}
    Let $R$ be a UFD, $K$ its field of fractions and $p=\mathrm{char}(K)$. 
    \begin{enumerate}
        \item[(i)] Assume $p=0$. Then $f\in R[X]$ is special iff it is of the form $$f(X)=f_d(X+A)^d+C,\quad 0\neq f_d\in R,\quad A,C\in K,$$
        where $d\ge 2$ is such that there exists a primitive $d$-th root of unity in $K$.

        \item[(ii)]
        Assume $p>0$. Then $f\in R[X]$ of degree $\deg f=d=p^l m\ge 2 , (m,p) = 1$ is special iff it is of the form  \[f(X) = f_d \prod_{i=1}^{p^v} (X - b_i + A)^{m p^{l-v}} + C,\] $$0\neq f_d\in R,\quad A, C \in K,\quad 0 \leq v \leq l,\quad \zeta\in K,\quad V = \{b_1,...,b_{p^v}\} \subset K,$$ where $\zeta$ is a primitive $m$-th root of unity and $V$ is an $\mathbb{F}_p(\zeta)$-linear subspace of $K$ with $|V|=p^v$.
   
    \end{enumerate}
\end{thrm}
The proof of \Cref{clasifaction of spesial polinomials} will be given in \Cref{sec: class special}.

We now briefly discuss how our main 
conjecture and results compare with the work
of Leumi \cite{leumi2021lcm}, which also 
studied the function field analog of 
Cilleruelo's conjecture and influenced the 
present work. First, \cite{leumi2021lcm} 
states Conjecture \ref{the conjecture} 
without the constant $c_f$. This is certainly 
false by Theorem \ref{Upper bound} because it 
can happen that $c_f<1$ (see Example 
\ref{ex:1} below). It seems to have been 
overlooked that when $g\in V_f$ and $Q\in 
M_n,\,n>\deg g$ then since $f(Q+g)=f(Q)$, the 
value $f(Q+g)$ contributes nothing new to the 
LCM on the RHS of (\ref{eq: def Lfn}) over 
the contribution of $f(Q)$. Once one accounts 
for this redundancy with the constant 
$c_f=1/|V_f|$, Cilleruelo's original 
heuristic carries over to the function field 
case giving rise to Conjecture \ref{the 
conjecture}. Second, all results and 
conjectures in \cite{leumi2021lcm} are stated 
only for an absolutely irreducible and 
separable $f\in\F_q[T][X]$. We do not make 
these restrictions here and it takes 
additional technical work to treat the 
general case. Third, the lower bound on 
$L_f(n)$ given in \cite{leumi2021lcm} has a smaller constant than ours 
(thus the bound is weaker), comparable to 
the RHS of (\ref{eq:sah_lower_bound_rad}), and a lower 
bound for the radical of $L_f(n)$ is not 
stated explicitly. Fourth, in 
\cite{leumi2021lcm} our Theorem \ref{special 
= conjecture} is stated without the constant 
$c_f$, which is incorrect in general for the 
same reasons explained above, although the 
arguments therein are essentially correct in 
the case $c_f=1$. Finally, 
\cite{leumi2021lcm} gives a classification of 
special polynomials only in the case $p\nmid d$ (and it is stated only for the ring $\F_q[T]$), whereas we treat the general case.

\textbf{Acknowledgments.} The authors would like to thank Ze\'ev Rudnick and the anonymous referee of this paper for spotting a few small errors in a previous draft of the paper. Both authors were partially supported by Israel Science Foundation grant no. 2507/19.

\section{Preliminaries}\label{sec:prelim}
For background on the arithmetic of function fields see \cite{rosen2002number}. For background on resultants, which we will use below, see \cite[\S IV.8]{Lang2002}.
\subsection{Notation}
We now introduce some notation which will be used throughout Sections \ref{sec:prelim}-\ref{sec: special asym}. Let $p$ be a prime, $q=p^k$. For $Q\in\F_q[T]$ we denote by $|Q|=q^{\deg Q}$ the standard size of $Q$. For $P\in\F_q[T]$ prime and $Q\in\F_q[T]$ we denote by $v_P(Q)$ the exponent of $P$ in the prime factorization of $Q$.

We will always fix a polynomial $$f(X)=\sum_{i=0}^df_iX^i\in\F_q[T][X],\,f_d\neq 0$$ of degree $d$.

We also adopt the following conventions about notation.
\begin{itemize}\item For a polynomial $Q\in\F_q[T]$ we denote by $\deg Q$ its degree in $T$. For a polynomial $g\in\F_q[T][X]$ we denote by $\deg g=\deg_Xg$ its degree in the variable $X$.
\item For a polynomial $g\in\F_q[T][X]$ we denote by $g'=\frac{\partial g}{\partial X}$ its derivative in the variable $X$. The derivative in the variable $T$ will be written explicitly $\frac{\partial g}{\partial T}$.
\item $g\in\F_q[T][X]$ is called separable if it is separable as a polynomial in the variable $X$, equivalently $g\not\in\F_q[T][X^p]$.
\item For two polynomials $g,h\in\F_q[T][X]$ we denote by $\mathrm{Res}(g,h)=\mathrm{Res}_X(g,h)$ their resultant in the variable $X$.
\end{itemize}

\subsection{The space $V_f$}

Recall that $V_f=\left\{g \in \A:\ f\left(X + g\right) = f\left(X\right)\right\}$.  

\begin{lem}
\label{V_f: trivial when degree is not devisable by p} Assume $d\ge 1$.

    \begin{enumerate}
        \item[(i)] $|V_f| \leq d$.
        \item[(ii)] $V_f$ an $\F_p$-linear subspace of $\F_q[T]$
        \item[(iii)] $|V_f|$ is a power of p.
        \item[(iv)] If $p \nmid d$ then $V_f$ is trivial.
        \item[(v)] If $f_d\nmid f_i$ for some $1\le i\le d-1$ then $|V_f|\le d-1$.
    \end{enumerate}
\end{lem}

\begin{proof}
    $ $
    \begin{enumerate}
        \item[(i)] Assume by way of contradiction that $|V_f| \geq d + 1$. Then $f(g)=f(0)$ for every $g\in V_f$. Since $f(x) - f(0)$ has at most $d$ roots, $f\left(x\right) = f\left(0\right)$ and $f$ is constant, a contradiction.
        \item[(ii)] It is obvious that $0 \in V_f$. Now let $a, b \in V_f$ and let us prove $\alpha a + \beta b \in V_f$ where $\alpha, \beta \in \{0, 1, ..., p - 1\}$. Recursively applying $f(X+a)=f(X+b)=f(X)$ we obtain
        \[f\left(X + \alpha a + \beta b\right) = f\left(X + \sum_{i=1}^{\alpha} a + \sum_{i=1}^{\beta} b\right) = f\left(X\right).\]
        \item[(iii)] Obvious from (i) and (ii).
        \item[(iv)] Let $g\in V_f$, so $f(X+g)=f(X)$. Comparing coefficients at $X^{d-1}$ we find $dg=0$. If $p\nmid d$ then $g=0$, so in this case $V_f=\{0\}$ and the claim follows.
        \item[(v)] If $|V_f|=d$ then since all elements $g\in V_f$ are roots of $f(X)-f(0)$, we have $f=f_d\prod_{g\in V_f}(X-g)+f(0)$ and $f_d|f_i$ for all $1\le i\le d-1$, contradicting the assumption.
    \end{enumerate}
\end{proof}

\begin{ex}
\label{ex:1}
    (Example of a polynomial with $|V_f| = d=\deg f$) Let $V = \{b_1,...,b_{p^v}\}\subset\F_q[T]$ be an $\F_p$-linear subspace and $C\in\F_q[T]$. Then the polynomial $$f(X)=\prod_{i=1}^{p^v}(X-b_i)+C$$ has $V_f = V$. Thus $|V_f| = |V| = p^v = d$.
\end{ex}

\subsection{Roots of $f$ modulo prime powers}

In this subsection we study the quantity
  \[\p{k} := \left|\{Q \bmod P^k: f(Q) \equiv 0 \bmod P^k\}\right|,\]
i.e. the number of roots of $f$ modulo a prime power $P^k$.

\begin{lem}
\label{shared roots mod P}
    Let $g, h \in \AX$ be polynomials and let $P\in\F_q[T]$ prime. If $g, h$ have a common root modulo $P^m$ then $P^m \mid R := \mathrm{Res}(g, h)$.
\end{lem}

\begin{proof} 
    We can express $R$ as $R = a(X)f(X) + b(X)g(X)$ for some $a(X), b(X) \in \AX$ (see \cite[\S IV.8]{Lang2002}). Therefore, if there exists a $Q \in \A$ such that $P^m \mid f(Q)$ and $P^m \mid g(Q)$, then $P^m$ must also divide $a(Q)f(Q) + b(Q)g(Q) = R$.
\end{proof}

The proof of the next two lemmas is similar to the analogous proof for the integer case in \cite[proof of Theorem II]{TrygveNagel1921}.

\begin{lem} \label{beta for Hensel}Assume $f$ is separable and denote $R := \mathrm{Res}_X(f, f')\neq 0$. Denote $\mu=v_P(R)$. Let $x_0,x_1\in\F_q[T]$ be such that the following conditions hold:
\begin{enumerate}
    \item[(a)] $f(x_0)\equiv 0\pmod{P^{\mu+1}}$.
    \item[(b)] $f(x_1)\equiv 0\pmod{P^{\beta+1}}$, where $\beta:=v_P(f'(x_0))$.
    \item[(c)] $x_1 \equiv x_0 \mod P^{\mu + 1}$.
\end{enumerate}
Then   
     $\beta \leq \mu$ and $v_P(f'(x_1))=\beta$.
\end{lem}

\begin{proof}
    Since $f(x_0) \equiv 0 \mod P^{\mu +1}$ and $P^{\mu + 1} \nmid R$, by Lemma \ref{shared roots mod P} we must have $P^{\mu + 1} \nmid f'(x_0)$. Hence $\beta \leq \mu$.
Now writing $x_1=x_0+tP^{\mu+1},\,t\in\F_q[T]$ and using $\beta\le\mu$ and conditions (a),(b) we have
    \[f'(x_1) = f'(x_0 + tP^{\mu + 1}) \equiv f'(x_0) \equiv 0 \pmod {P^{\beta}},\]
    \[f'(x_1) = f'(x_0 + tP^{\mu + 1}) \equiv f'(x_0) \not\equiv 0 \pmod{P^{\beta + 1}},\]
so $v_P(f'(x_1))=\beta$.
\end{proof}

\begin{lem}
    \label{Lift root seprable}
    In the setup of Lemma \ref{beta for Hensel} let $\alpha>\mu$ be an integer and assume that in fact $f(x_1)\equiv 0\pmod{P^{\alpha+\beta}}$. Consider the set
    $$S_1 := \left\{ x_1 + u P^{\alpha} \mid u \in \A / P^{\beta}\right\}\subset\F_q[T]/P^{\alpha+\beta}.$$ Then
    \begin{enumerate}
    \item[(i)] The elements of $S_1$ are roots of $f$ modulo $P^{\alpha+\beta}$.
    \item[(ii)] The number of roots of $f$ modulo $P^{\alpha+\beta + 1}$ that reduce modulo $P^{\alpha+\beta}$ to an element of $S_1$ is equal to $|S_1|=q^{\beta \deg P}$.
    \end{enumerate}
\end{lem}

\begin{proof}
    To prove (i) we note that
    \[f(x_1 + uP^{\alpha}) \equiv f(x_1) + uP^{\alpha}f'(x_1) \pmod{P^{2\alpha}}.\]
    Thus $P^{\alpha+\beta} \mid f(x_1 + uP^{\alpha})$ as $P^{\alpha+\beta} \mid f(x_1)$ and by Lemma \ref{beta for Hensel} we have $P^{\beta} \mid f'(x_1)$ and $\alpha+\beta \le\alpha+\mu< 2\alpha$.

    To show (ii), consider the set of possible lifts from $S_1$ to $\F_q[T]/P^{\alpha+\beta+1}$, i.e. \[S_2 := \left\{x_1 + u P^{\alpha} + v P^{\alpha + \beta} \mid u \in \A / P^{\beta}, v \in \A/P\right\}.\]

   We will now determine for which 
   $u,v$ the element 
   $x_1+uP^\alpha+vP^{\alpha+\beta}$ 
   is a root of $f$ modulo 
   $P^{\alpha+\beta+1}$. Using 
   $2\alpha>\alpha+\beta$ we have
    \[f(x_1 + u P^{\alpha} + v 
    P^{\alpha + \beta}) \equiv  
    f(x_1) + u P^{\alpha} f'(x_1)+vP^{\alpha+\beta}f'(x_1) 
    \pmod{ P^{\alpha + \beta + 1}}.\]
    Hence $x_1 + u P^{\alpha} + v 
    P^{\alpha + \beta}$ is a root of 
    $f\bmod P^{\alpha + \beta + 1}$ iff \begin{equation}\label{eq: condition for hensel}f(x_1) + u P^{\alpha} 
    f'(x_1)+vP^{\alpha+\beta}f'(x_1) \equiv 0 \pmod {P^{\alpha + 
    \beta + 1}}.\end{equation} As $P^{\alpha + 
    \beta} \mid f(x_1)$ and 
    $v_P(f'(x_1))=\beta$ (Lemma 
    \ref{beta for Hensel}), we have 
    (\ref{eq: condition for hensel}) iff
    \[ u \equiv - \left(\frac{f'(x_1)}{P^{\beta}}\right)^{-1} \left(\frac{f(x_1)}{P^{\alpha + \beta}}+vf'(x_1)\right) \pmod P\]
    Thus we have $|P|$ possible values of $v$ and for each of these $|P^{\beta}|/|P|$ possible values of $u$. Overall we have $|P^\beta|=q^{\beta\deg P}=|S_1|$ possible values of $(u,v)$ and the assertion follows.
\end{proof}

\begin{lem}
    \label{Hensel for seprable}
    Assume that $f$ is separable, $R=\mathrm{Res}(f, f')\neq 0$. Let $P$ be a prime in $\A$ and $\mu=v_P(R)$. Then $\p{2\mu + k} = \p{2\mu + 1}$ for all $k \geq  1$. 
\end{lem}

\begin{proof}
    For each root $x_1$ of $f$ modulo $P^{2\mu+k}$ we can apply Lemma \ref{Lift root seprable} with $\alpha=2\mu+k-\beta$, where $\beta=v_P(f'(x_1))$ (the condition $\alpha>\mu$ holds because $\beta\le\mu$ by Lemma \ref{beta for Hensel} applied with $x_0=x_1$) and obtain that the number of roots of $f$ modulo $P^{2\mu+k+1}$ equals the number of roots of $f$ modulo $P^{2\mu+k}$, i.e. $\p{2\mu + k+1} = \p{2\mu + k}$. The assertion now follows by induction on $k$. 
\end{proof}

\begin{lem}
    \label{non zero resultant}
    Assume $f\in\F_q[T][X^p]$ is inseparable and irreducible. Then $U := \mathrm{Res}(f, \frac{\partial f}{\partial T}) \neq 0$.
\end{lem}
\begin{proof}
    Assume by way of contradiction that $U = 0$. Then $f,\frac{\partial f}{\partial T}$ have a common factor and since $f$ is irreducible we have $f\mid\frac{\partial f}{\partial T}$. Comparing degrees in $T$ we must have $\frac{\partial f}{\partial T}=0$. This means that $f\in\F_q[T^p,X^p]$ is a $p$-th power, contradicting its irreducibility.
\end{proof}

\begin{lem}
\label{Hensel for inseprable}
    Assume that $f$ is inseparable and irreducible, and $P^m \nmid U := \mathrm{Res}(f, \frac{\partial f}{\partial T}) \in \A $ for some $P\in\F_q[T]$ prime and $m \geq 1$. Then $\p{k} = 0$ for every $k \geq m+1$. 
\end{lem}
\begin{proof}
    Assuming the existence of $Q \in \A$ such that $P^m \mid f(Q)$ (if $f$ has no roots modulo $P^m$ we are done), we will now prove that $P^{m+1} \nmid f(Q)$. Since $P^m \mid f(Q)$, we know that $P^m \nmid \frac{\partial f}{\partial T}(Q)$; otherwise, by Lemma \ref{shared roots mod P} we would have $P^m|U$, contradicting our assumption.

    Since $f$ is inseparable and irreducible we have $f'=0$ and therefore 
    \[\frac{\partial f(Q)}{\partial T} = \frac{\partial f}{\partial T}(Q) + \frac{\partial f}{\partial X}(Q) \frac{dQ}{dT} = \frac{\partial f}{\partial T}(Q).\]
    If $P^{m+1} \mid f(Q)$, write $f(Q) = P^{m+1} C$ and then
    \begin{equation*}
        \begin{split}
           \frac{\partial f(T,Q(T))}{\partial T} &= \frac{ \partial(P^{m+1} C)}{\partial T} = P^{m+1}  \frac{\partial C}{\partial T} + (m+1)P^{m}  \frac{\partial P}{\partial T}  C. \\ 
        \end{split}
    \end{equation*}
    Thus $P^m \mid \frac{\partial f(T,Q(T))}{\partial T} = \frac{\partial f}{\partial T}(Q)$, contradicting the above observation $P^m \nmid \frac{\partial f}{\partial T}(Q)$.
\end{proof}

\begin{prop}
    \label{bound on pn}
    Assume that $f$ is irreducible. Then $\p{m} \ll 1$ for all $P$ prime and $m \geq 1$.
\end{prop}

\begin{proof}
    Since $\A / P$ is a field we have $\p{} \leq d$. Thus it remains to handle the case $m\ge 2$.

    If $f$ is inseparable then using Lemmas \ref{Hensel for inseprable} and \ref{non zero resultant} we see that there are only finitely many pairs $P,m$ such that $m \geq 2$ and $\p{m} \neq 0$ (as they must satisfy $P^m\,|\,\mathrm{Res}(f,\frac{\partial f}{\partial T})\neq 0$).

    Now if $f$ is separable then by Lemma \ref{Hensel for seprable} there are only finitely many pairs of $(P, m)$ such that $m \geq 2$ and $\p{m} \neq \p{m-1}$, Denote these pairs by $(P_i, m_i)_{i=1}^s$. Then for all primes $P$ and $m \geq 1$
    \[\p{m} \leq \max\{d, \rho_f(P_i^{m_i})_{i=1}^s\}\ll 1.\]
\end{proof}

\begin{lem}
\label{Replacment to inseparable polinomial}
    Let $f \in \F_q[T][X^p]$ be inseparable and irreducible. Then there exist a separable and irreducible $h \in \AX$ such that $\p{} = \rho_h(P)$ for all primes $P \in \A$.
\end{lem}

\begin{proof}
    Write $f(X) = h(X^{p^m})$ for some $m \geq 1$ such that $h\in\F_q[T][X]\setminus\F_q[T][X^p]$ is separable. For any prime $P$, the $m$-fold Frobenius automorphism of $\F_q[T]/P$ given by $x\mapsto x^{p^m}$ gives a one-to-one correspondence between the roots of $f$ modulo $P$ and the roots of $h$ modulo $P$. Hence $\rho_f(P) = \rho_h(P)$.
\end{proof}

\section{The upper bound}\label{sec: upper bound}

Throughout this section the polynomial $f$ is assumed \emph{irreducible}.
To obtain the upper and lower bounds on $\LL$ (recall that this quantity was defined in (\ref{eq: def Lfn})) we set $$P_f(n) : = \prod_{Q \in M_n} f(Q)$$ and write 
\begin{equation}\label{eq: L P}\deg L_f(n) = c_f \deg P_f(n)- (c_f \deg P_f(n) - \deg L_f(n)),\end{equation}
where $c_f=1/|V_f|$.
\begin{lem}
\label{P_f(n)}
    $\deg P_f(n) = dnq^n + O(q^n)$.
\end{lem}
\begin{proof}
    For sufficiently large $n$ and $Q \in M_n$ we have $\deg f(Q) = dn + \deg f_d$, hence
    \begin{equation*}
            \deg \PP  = \deg \prod_{Q \in M_n} f(Q) = \sum_{Q \in M_n} \deg f(Q) 
             = \sum_{Q \in M_n} (dn + \deg f_d) = d nq^n + O(q^n).
    \end{equation*}
\end{proof}

\noindent {\bf Convention.} Throughout the rest of the paper $P$ will always denote a prime of $\F_q[T]$ and $\sum_P$ (resp. $\prod_P$) will denote a sum (resp. product) over all monic primes of $\F_q[T]$. A sum of the form $\sum_{a\le\deg P\le b}$ is over all monic primes in the corresponding degree range (and the same for products).
\\

To estimate $c_f \deg P_f(n)  - \deg L_f(n)$ we write the prime decomposition of $\LL$ and $\PP$ as
\[ \LL = \prod_{P} P^{\beta_P(n)},\quad \PP = \prod_{P} P^{\alpha_P(n)},\]
where the products are over all (monic) primes in $\F_q[T]$.
\\ \\
{\bf Convention.} Throughout Sections \ref{sec: upper bound}-\ref{sec: special asym} we will always assume that $n$ is large enough so that $f(Q)\neq 0$ for all $Q\in M_n$. Thus $L_f(n),P_f(n)\neq 0$ and $\alpha_P(n),\beta_P(n)$ are always finite.
\\

We have
\begin{equation}\label{eq:def alpha beta}\an = \sum_{Q \in M_n} \val{f(Q)} ,\ \bn = \max\{ \val{f(Q)} : Q \in M_n\},\end{equation}
(recall that $\val{Q}$ is the exponent of $P$ in the prime factorization of $Q$). Combining (\ref{eq: L P}) with \Cref{P_f(n)} we have
\[\deg \LL = c_f d nq^n - \sum_{P}(c_f \an - \bn) \deg P + O(q^n).\]

\begin{lem}
\label{bound on bn using an}
    For sufficiently large n 
    \[\bn \leq c_f \an .\]
\end{lem}
\begin{proof}
    Denote by $Q_{max} \in M_n$ an element such that $\bn = v_P(f(Q_{max}))$. Then for each element $g \in V_f$ we have 
    $f(Q_{max} + g) = f(Q_{max})$ and hence $\beta_P(n)= v_P(f(Q_{max})) = v_P(f(Q_{max} + g)).$
    For $n$ sufficiently large $Q_{max} + g \in M_n$, so
    \[|V_f| \bn = \sum_{g \in V_f} v_P(f(Q_{max} + g)) \leq \sum_{Q \in M_n} v_P(f(Q)) = \an .\]
    The assertion follows since $c_f=1/|V_f|$.
\end{proof}

The next lemma is the main tool for estimating $\alpha_P(n), \beta_P(n)$. For its proof we introduce the following notation, which will be used in the sequel as well:
\begin{equation}\label{eq: def sf}s_f(P^k,n) = \left|\{Q \in M_n: f(Q) \equiv 0 \bmod P^k\}\right|,\end{equation}
where $P$ is prime and $k\ge 1$.

\begin{lem}
    \label{estimate an and bn}
    Let $P \in \A$ be a prime.
    
    \begin{enumerate}
        \item[(i)] $\bn = O\left(\frac{n}{\deg P}\right)$.
        \item[(ii)] If $f$ is separable and $P \nmid \mathrm{Res}(f,f')$, then 
        \[\an = q^n \frac{\p{}}{|P| - 1} + O\left(\frac{n}{\deg P}\right).\]
        \item[(iii)] If $f$ is inseparable and $P \nmid \mathrm{Res}(f, \frac{\partial f}{\partial T})$, then
        \[\an = q^n \frac{\p{}}{|P|} + O\left(\frac{n}{\deg P}\right).\]
        \item[(iv)] For the finitely many ``bad" primes where neither (ii) nor (iii) hold we have
        \[\an = O(q^n).\]
    \end{enumerate}
\end{lem}
\begin{proof}
    To prove (i) we note that since there exists $Q \in M_n$ such that $P^{\bn} \mid f(Q)$ and $f(Q) \neq 0$ we get:
    \[\bn \deg P \leq \deg f(Q).\]
    For sufficiently large $n$ we have $\deg f(Q) = dn + \deg f_d$, thus
    \[\bn \ll \frac{n}{\deg P},\] establishing (i).
    
    Using the notation (\ref{eq: def sf}) we have
    $$
           \an  = \sum_{Q \in M_n} v_P(Q) = \sum_{Q \in M_n} \sum_{\substack{k \geq 1  \atop{P^k \mid f(Q)}}} 1 
            = \sum_{k \geq 1{ }} \sum_{ \substack{ Q \in M_n  \atop{P^k \mid f(Q)}}} 1 = \sum_{k \geq 1} s_f(P^k,n).
    $$
    Note that $$s_f(P^k,n) = q^n\frac{\p{k}}{|P^k|} + O(1),$$ since if $Q$ is a root of $f$ modulo $P^k$ and $\deg P^k \le n$ then there are exactly $\frac{q^n}{|P|^k}$ elements in $M_n$ that reduce to $Q$ modulo $P^k$. And if $\deg P^k > n$ then $s_f(P^k,n) \leq \p{k} \ll 1$ by Proposition \ref{bound on pn}. Hence
    \begin{equation*}
           \an = \sum_{k \geq 1} s_f(P^k,n) = \sum_{k \geq 1} \left[q^n\frac{\p{k}}{|P|^k} + O(1)\right]
           = q^n \sum_{k \geq 1} \frac{\p{k}}{|P|^k} + O\left(\frac{n}{ \deg P}\right).
    \end{equation*}
    
    Let us look at the different cases:
    \begin{itemize}
        \item If $f$ is separable and $P \nmid \mathrm{Res}(f,f')$ then we have by \Cref{Hensel for seprable} that $\p{k} = \p{}$, thus
        
        \begin{equation*}
            \begin{split}
                \an & = q^n \sum_{k \geq 1} \frac{\p{}}{|P|^k} + O\left(\frac{n}{ \deg P}\right)\\ 
                & = q^n \p{} \sum_{k \geq 1} \left(\frac{1}{|P|}\right)^k + O\left(\frac{n}{ \deg P}\right) \\
                & = q^n \frac{\p{}}{|P| - 1} + O\left(\frac{n}{ \deg P}\right).
            \end{split}
        \end{equation*}
        
        \item If $f$ is inseparable and $P \nmid \mathrm{Res}_X(f, \frac{\partial f}{\partial T})$ then using Lemma \ref{Hensel for inseprable} we get
        \[\an = q^n \frac{\p{}}{|P|} + O\left(\frac{n}{\deg P}\right).\]
        
        \item For a general irreducible $f$, using the fact that $\text{for every } k \geq 1,\ \p{k} \ll 1$ (\Cref{bound on pn}) we get
        \begin{equation*}
            \begin{split}
              \an   = q^n \sum_{k \geq 1} \frac{\p{k}}{|P|^k} + O\left(\frac{n}{ \deg P}\right)
               \ll q^n \sum_{k \geq 1} \frac{1}{|P|^{k}} + O\left(\frac{n}{ \deg P}\right) 
               \ll O(q^n).
            \end{split}
        \end{equation*}
    \end{itemize}
\end{proof}

\begin{lem}\label{lem: medium range}
    \[\deg\left(\prod_{n< \deg P \leq n + \deg f_d} P^{\an}\right) = O(q^n).\]
\end{lem}
\begin{proof}
    By Lemma \ref{estimate an and bn} for sufficiently large $n$ and $P$ a prime such that $\deg P > n$, we have $\an \leq  q^n \frac{\p{}}{|P| - 1} + O\left(\frac{n}{ \deg P}\right) \ll \p{} + O(1)  \ll 1$. Using the Prime Polynomial Theorem, we have
    \begin{equation*}
        \begin{split}
            \deg\left(\prod_{n< \deg P \leq n + \deg f_d} P^{\an}\right) & = \sum_{ n < \deg P \leq n + \deg f_d} \an \deg P \\
            & \ll \sum_{n < \deg P \leq n + \deg f_d} \deg P = \sum_{k=n+1}^{n + \deg f_d} \sum_{\deg P = k} k \\
            & \ll \sum_{k=n+1}^{n + \deg f_d}  k \frac{q^k}{k} = \frac{q^{n + \deg f_d + 1} - q^{n+1}}{q - 1} \\
            & \ll q^n.
        \end{split}
    \end{equation*}
\end{proof}

\begin{prop}
    \label{small P}
    Denote $R_f(n) = \prod_{\deg P \leq n + \deg f_d} P^{\an}$. Then
    \[\deg R_f(n) = n q^n + O(q^n).\]
\end{prop}
\begin{proof}
    Let us assume first that $f$ is separable. We will handle the inseparable case at the end of the proof. We note that we may ignore the $O(1)$ \textquotedblleft bad" primes (as defined in Lemma \ref{estimate an and bn}) with an error term of $O(q^n)$ and by \Cref{lem: medium range} we can ignore the primes with $n<\deg P\le n+\deg f_d$ with the same error term. Thus by \Cref{estimate an and bn} we obtain
    \begin{equation}\label{eq: R_f(n)}
        \begin{split}
            \deg R_f(n)  = \sum_{\deg P \leq n} \an \deg P + O(q^n)  
             = \sum_{\deg P \leq n} q^n \frac{\rho_f(P)}{|P| - 1} \deg P + \sum_{\deg P \leq n} O(n) + O(q^n).
        \end{split}
    \end{equation}
    
We bound the error term using the Prime Polynomial Theorem:
    \begin{equation}\label{eq: error in R_f(n)}
    \begin{split}
        \sum_{\deg P \leq n} n = n \sum_{k=1}^n \sum_{\deg P = k} 1 
         \ll n \sum_{k=1}^n \frac{q^k}{k}  \ll q^n.
    \end{split}
    \end{equation}
    
    Now to estimate $\sum_{\deg P \leq n} q^n \frac{\rho_f(P)}{|P| - 1} \deg P$ we will use the Chebotarev Density Theorem in function fields \cite[Proposition 7.4.8]{Fried2023}. First we introduce some notation and recall some terminology. Let $E / \F $ be the splitting field of $f$ and denote by $G$ the Galois group of $f$. For each prime $P$ of $\A$ unramified in $E/\F_q(T)$ the Frobenius class of $P$ is defined to be:
    \begin{equation*}
    \begin{split}
        \forb{P} = \left\{\sigma \in G \,:\, \exists \ \mathfrak{P}/P \text{ prime of } E\text{ such that } x^{|P|}\equiv \sigma(x)\pmod{\mathfrak P}\text{ for all }x\text{ with }v_{\mathfrak P}(x)\ge 0\right\} .
    \end{split}
    \end{equation*}
    The Frobenius class $\forb{P}$ is a conjugacy class in $G$. Denote by $S$ the set of all conjugacy classes in $G$. Given a conjugacy class $C \in S$, we set
    \[\pi_C(n) = \left| \left\{P\ \mathrm{prime} : \deg P = n, \forb{P} = C\right\}\right| . \]
    Let $K = \mathbb{F}_{q^v}$ be the algebraic closure of $\mathbb{F}_q$ in $E$. We have $G_0 := \mathrm{Gal}(K / \mathbb{F}_q) = \langle \phi \rangle$, where $\phi(x)=x^q$ is the $q$-Frobenius. Denote the restriction of automorphisms map by $\varphi: G \to G_0$. Since $G_0=\langle\phi\rangle$ is cyclic and in particular abelian, we have for all $C \in S$ that $\varphi (C) = \{\phi^{n_C}\}$ for some $n_C\in\Z$. Define $$S_k := \{C \in S \mid \varphi(C) = \{\phi^k\}\}.$$
    
    Now the Chebotarev Density Theorem in function fields says that
    \[\pi_C(k) = \left\{\begin{array}{ll}v \frac{|C|}{|G|} \frac{q^k}{k} + O\left(\frac{q^{k/2}}{k}\right),&C\in S_k,\\ 0,&C \notin S_k.\end{array}\right.\]
    
    Since the Galois group acts on the roots of $f$, we can define $\mathrm{Fix}(C)$ to be the number of roots fixed by any element in the conjugacy class $C$ (this number is the same for all $\sigma \in C$). Assuming $P$ is unramified in $E/\F_q(T)$, we have $\p{} = \mathrm{Fix}\left(\forb{P}\right)$. 
    In the calculations throughout the rest of the proof summation over $P$ denotes summation over primes $P\nmid \mathrm{Res}_X(f,f')$ if $f$ is separable and $P\nmid\mathrm{Res}(f,\frac{\partial f}{\partial T})$ otherwise. In the case when $f$ is separable this conditions insures that $P$ is unramified in $E/\F_q(T)$. As we observed above the excluded primes contribute $O(q^n)$.
    
    We have:
    \begin{multline*}
            \sum_{\deg P = k} \frac{\p{}}{q^k - 1}k = \sum_{\deg P = k} \frac{\mathrm{Fix}\left(\forb{P}\right)}{q^k - 1} k 
            = \sum_{C \in S}\frac{\mathrm{Fix}(C) \pi_C(k)}{q^k - 1} k  
            = \sum_{C \in S_k} \left[\frac{\mathrm{Fix}(C) q^k}{(q^k - 1)k} \frac{v|C|}{|G|} k + O(q^{-k/2}) \right] \\
            = \sum_{C \in S_k} v \frac{\mathrm{Fix}(C) |C|}{|G|}  + O(q^{-k/2})
        \end{multline*}
    and therefore
    \begin{multline}\label{eq: calc cheb}
             \sum_{\deg P \leq n} q^n \frac{\p{}}{q^k - 1}k  = q^n \sum_{k=1}^{n}  \sum_{\deg P = k} \frac{\p{}}{q^k - 1}k 
              = q^n \sum_{k=1}^{n} \left[ \sum_{C \in S_k} v \frac{\mathrm{Fix}(C) |C|}{|G|} + O(q^{-k/2}) \right]\\
             = q^n \sum_{k=1}^{n} \sum_{C \in S_k} v \frac{\mathrm{Fix}(C) |C|}{|G|} + O(q^n)
             = q^n \sum_{l=1}^{\lfloor \frac{n}{v} \rfloor} v \sum_{C \in S} \frac{\mathrm{Fix}(C) |C|}{|G|} + O(q^n).
    \end{multline}
    Using Burnside’s lemma \cite[Theorem 3.22]{rotman2012introduction} and the transitivity of $G$ (which is a consequence of $f$ being irreducible) we get
    \[1 = \sum_{\sigma \in G} \frac{\mathrm{Fix}(\sigma)}{|G|} = \sum_{C \in S} \frac{\mathrm{Fix}(C)|C|}{|G|} . \]
    Pluggin this into (\ref{eq: calc cheb}) and recalling (\ref{eq: R_f(n)}),(\ref{eq: error in R_f(n)}) we obtain
    \[\deg R_f(n) = \sum_{\deg P \leq n} q^n \frac{\rho_f(P)}{|P| - 1} \deg P + O(q^n) = q^n \sum_{l=1}^{\lfloor \frac{n}{v} \rfloor} v + O(q^n) = nq^n + O(q^n). \]
    
    This completes the proof in the separable case. To handle the case when $f$ is inseparable we again do the same calculations (using Lemma \ref{estimate an and bn}(iii) this time):
    
    \begin{equation*}
        \begin{split}
            \deg R_f(n) & = \sum_{\deg P \leq n} \an \deg P + O \left ( q^n \right) \\ 
            & = \sum_{\deg P \leq n} q^n \frac{\rho_f(P)}{|P|} \deg P + \sum_{\deg P \leq n} O(n) + O(q^n).
        \end{split}
    \end{equation*}
    
    We can replace $\rho_f(P)$ with $\rho_h(P)$, where $h$ is the polynomial given by Lemma \ref{Replacment to inseparable polinomial}. Since $h$ is separable, the argument above (with $f$ replaced by $h$) yields:
    \begin{equation*}
    \deg R_f(n) = \sum_{\deg P \leq n} q^n \frac{\rho_h(P)}{|P|} + O(q^n)=nq^n+O(q^n),
    \end{equation*}
    completing the proof.
\end{proof}

We are now ready to prove the upper bound on $\deg L_f(n)$.
\begin{proof}[Proof of Theorem \ref{Upper bound}]
    Using Lemma \ref{bound on bn using an} we have $c_f\an \leq \bn$, for all $P$ prime. We have
    \begin{align}
       \deg \LL & = c_f \deg \PP - (c_f \deg \PP - \deg \LL) = \nonumber \\
       & = c_f d n q^n - \sum_P (c_f\an - \bn ) \deg P \nonumber \\
       & \leq c_f d n q^n - \sum_{\deg P \leq n + \deg f_d} (c_f \an - \bn ) \deg P \label{eq: proof upper 1}\\
       & = c_f d n q^n - c_f\sum_{\deg P \leq n + \deg f_d} \an \deg P + \sum_{\deg P \leq n+\deg f_d} \bn \deg P \nonumber \\
       &= c_f(d-1) n q^n + \sum_{\deg P \leq n + \deg f_d} \bn \deg P + O(q^n)\label{eq: proof upper 2}
    \end{align}
    where (\ref{eq: proof upper 1}) comes from removing negative terms and (\ref{eq: proof upper 2}) from Proposition \ref{small P} and Lemma \ref{P_f(n)}.
    
    It remains to prove that $\sum_{\deg P \leq n+\deg f_d} \bn \deg P = O(q^n)$. From Lemma \ref{estimate an and bn}(i) and the Prime Polynomial Theorem we get
    \begin{equation}\label{eq: bound small primes}
            \sum_{\deg P \leq n + \deg f_d} \bn \deg P  \ll \sum_{\deg P \leq n + \deg f_d} \frac{n}{\deg P} \deg P 
            = \sum_{\deg P \leq n + \deg f_d} n \ll q^n.
    \end{equation}
    
\end{proof}

\section{Heuristic justification for \Cref{the conjecture}}\label{sec: heuristic}

Throughout the present section we maintain the assumption that $f=\sum_{i=0}^df_iX^i\in\F_q[T][X]$ is irreducible. Next we want to study the difference $\deg L_f(n)-c_f(d-1)nq^n$ and we will do so by relating it to a slightly different quantity defined by (\ref{eq: def Sf}) below. To this end let us define an equivalence relation on $M_n$ by
\[\ Q_1 \sim Q_2 \iff f(X + Q_1 - Q_2) = f(X)\iff Q_1-Q_2\in V_f.\]
We note that for sufficiently large $n$ the size of each equivalence class is $|V_f|$ and the number of such classes is $c_fq^n$.
Consider
    \begin{equation}\label{eq: def Sf}
        S_f(n) = \left| \left\{P \in \A \mbox{ prime}: \begin{aligned} \deg P >& n + \deg f_d.\\ \exists\, Q_1 \nsim Q_2 \in M_n& \text{ such that } P \mid f(Q_1), f(Q_2)\end{aligned} \right\}\right|.
    \end{equation}

Note that the condition in the bottom line on the RHS of (\ref{eq: def Sf}) is equivalent to $\beta_P(n)\neq c_f\alpha_P(n)$ (the negation of this condition is that $P$ occurs precisely in all $f(Q)$ where $Q$ runs over a single equivalence class of $\sim$, equivalently $|V_f|\beta_P(n)=\alpha_P(n)$), hence \begin{equation}\label{eq: exp for Sf} S_f(n) = \sum_{\deg P>n + \deg f_d\atop{c_f\alpha_P(n)\neq\beta_P(n)}}1. \end{equation}

\begin{prop}
    \label{Eqiv to conjecture}
    Let $f \in \AX$ be irreducible with $\deg_X f = d \geq 2$.
    \begin{enumerate}
        \item[(i)] $S_f(n)\ll q^n.$
        \item[(ii)] \Cref{the conjecture} is equivalent to $S_f(n) = o(q^n).$
    \end{enumerate}
\end{prop}

\begin{proof}   
    The first part follows immediately from the definition of $S_f(n)$, since there are $q^n$ possible values of $f(Q)$ and each has $O(1)$ prime factors of degree $\deg P>n+\deg f_d$. It remains to prove the second part.

    In the main calculation in the proof of Theorem \ref{Upper bound} there is only one inequality, namely
    $$
        c_f d n q^n - \sum_P (c_f\an - \bn ) \deg P  \leq c_f d n q^n - \sum_{\deg P \leq n + \deg f_d} (c_f \an - \bn ) \deg P.
    $$
    equivalently
    \[ - \sum_{\deg > n + \deg f_d} (c_f \an - \bn ) \deg P \leq 0.\]
    Since the RHS in Conjecture \ref{the conjecture} is $\gg nq^n$, the conjecture is now seen to be equivalent to $$\sum_{\deg > n + \deg f_d} (c_f \an - \bn ) =  o(n q^n).$$ Hence is suffices to prove that
    \begin{equation}\label{eq: need to prove Sfn}S_f(n) = o(q^n) \iff \sum_{\deg P > n + \deg f_d} (c_f \an - \bn) \deg P = o(n q^n).\end{equation}

    For $P \in \A$ prime set $$\delta_P(n) = \begin{cases} 1, & c_f \an \neq \bn,\\ 0, & \text{otherwise}.\end{cases}$$
    Note that for $n$ sufficiently large if $\deg P > dn + \deg f_d $ then $ c_f \an = \bn = 0$. Hence, for $n$ sufficiently large using (\ref{eq: exp for Sf}) we have
    \[S_f(n) = \sum_{n + \deg f_d < \deg P \leq dn + \deg f_d} \delta_P(n) \]
    From \Cref{estimate an and bn}(ii-iii) and \Cref{bound on pn}, for sufficiently large $n$ and $P$ not one of $O(1)$ bad primes, if $\deg P > n$ then $\an \ll 1$, so
    \begin{multline}\label{eq:Sfn 1}
            \sum_{n + \deg f_d < \deg P} (c_f \an - \bn) \deg P  \ll \sum_{n + \deg f_d < \deg P \leq d n+\deg f_d} \delta_P(n) \deg P \\
              \ll \sum_{n + \deg f_d < \deg P \leq d n + \deg f_d} \delta_P(n) n
             = S_f(n) n.
    \end{multline}
    On the other hand, once again using the fact that $c_f\alpha_P(n)=\beta_P(n)$ if $\deg P>dn+\deg f_d$,
    \begin{multline}\label{eq:Sfn 2}
            \sum_{n + \deg f_d < \deg P} (c_f \an - \bn) \deg P  \\  \geq \sum_{n + \deg f_d < \deg P \leq d n + \deg f_d} \delta_P(n) \deg P 
               \geq \sum_{n + \deg f_d < \deg P \leq d n + \deg f_d} \delta_P(n) n =  S_f(n) n,
    \end{multline}
From (\ref{eq:Sfn 1}) and (\ref{eq:Sfn 2}) the equivalence (\ref{eq: need to prove Sfn}) follows, which completes the proof.
\end{proof}

Heuristically one can estimate $S_f(n)$ by arguing that the ``probability" that for $Q_1\not\sim Q_2\in M_n$ the values $f(Q_1),f(Q_2)$ share a prime factor of degree $\deg P>n+\deg f_d$ is (using the Prime Polynomial Theorem)
$$\sum_{m=n+\deg f_d+1}^{\infty}\sum_{\deg P=m}\frac 1{|P|^2}=\sum_{m=n+\deg f_d+1}^{\infty}\frac 1{mq^m}=
O\left(\frac 1{nq^n}\right).$$ Since there are $O(q^{2n})$ pairs $Q_1,Q_2$ we expect $$S_f(n)\ll \frac 1{nq^n}\cdot q^{2n}\ll\frac{q^n}n=o(q^n).$$ By \Cref{Eqiv to conjecture} this is equivalent to \Cref{the conjecture}. 

At this point one may ask whether the equivalence relation $\sim$ is really the dominant source of pairs $Q_1,Q_2\in M_n$ with $f(Q_1)=f(Q_2)$, otherwise we would definitely not have $S_f(n)=o(q^n)$. Over the integers it cannot happen that $f(n_1)=f(n_2)$ for $n_1\neq n_2\in\Z$ sufficiently large (since a polynomial is a monotone function at sufficiently large arguments), but in finite characteristic this is less obvious. Nevertheless this is guaranteed by the following proposition, which will be needed for the proof of \Cref{special = conjecture} as well.

\begin{prop}\label{prop: fq1=fq2} $$\biggl|\biggl\{(Q_1,Q_2)\,:\,Q_1\not\sim Q_2\in M_n,\,f(Q_1)=f(Q_2)\biggr\}\biggr|=O(q^{n/2}). $$\end{prop}

Before we prove \Cref{prop: fq1=fq2} we need a couple of auxiliary results.

\begin{lem}\label{lem: hit}
Let $g(X,Y)\in\F_q[T][X,Y]$ be a fixed irreducible polynomial with $\deg_{X,Y} g\ge 2$. The number of pairs $Q_1,Q_2\in M_n$ such that $g(Q_1,Q_2)=0$ is $O(q^{n/2})$.
\end{lem}

\begin{proof} If $\deg_Xg=\deg_Yg=1$ then $g=AXY+BX+CY+D,\,A\neq 0,B,C,D\in\F_q[T]$ and one cannot have $g(Q_1,Q_2)=0$ for $Q_1,Q_2\in M_n, n>\max(\deg B,\deg C,\deg D)$ since the term $AQ_1Q_2$ has higher degree than the other terms. Hence assume WLOG that $\deg_Yg\ge 2$, otherwise switch $X$ and $Y$. It follows from the quantitative Hilbert Irreducibility Theorem in function fields \cite[Theorem 1.1]{BaEn21} that for all but $O(q^{n/2})$ polynomials $Q_1\in M_n$, the polynomial $g(Q_1,Y)\in\F_q[T][Y]$ is irreducible of degree $\ge 2$ and therefore has no root $Q_2\in M_n$. For the remaining $O(q^{n/2})$ values of $Q_1$ there are at most $\deg_Yg=O(1)$ roots for each $Q_1$, so overall there are $O(q^{n/2})$ pairs $Q_1,Q_2\in M_n$ with $g(Q_1,Q_2)=0$.
\end{proof}

\begin{lem} \label{lem: zeta} Assume that the polynomial $f(X)-f(Y)\in\F_q[T][X,Y]$ is divisible by $aX+bY+c$ for some $a,b,c\in\F_q[T]$ with $b\neq 0$. Then
\begin{enumerate}
    \item[(i)] $\zeta:=-a/b$ satisfies $\zeta^d=1$ and $\zeta\in\F_q$.
    \item[(ii)] If $\zeta=1$ then $c/b\in V_f$.
\end{enumerate}
\end{lem}

\begin{proof}
    The divisibility assumption implies $f(X)=f\left(-\frac ab X-\frac cb\right)$. Comparing coefficients at $X^d$ gives $\zeta^d=1$. Since $\zeta\in\F_q(T)$ and $\F_q$ is algebraically closed in $\F_q(T)$ we have $\zeta\in\F_q$, establishing (i). If $\zeta=-1$ the above identity becomes $f(X)=f\left(X-\frac cb\right)$, i.e. $-c/b\in V_f$. Since $V_f$ is a vector space we obtain (ii).
\end{proof}

We are ready to prove \Cref{prop: fq1=fq2}.

\begin{proof}[Proof of \Cref{prop: fq1=fq2}] Write $f(X)-f(Y)=c\prod_{i=1}^m g_i(X,Y)$ with $c\in\F_q[T]$ and each $g_i\in\F_q[T][X,Y]$ irreducible with $\deg_{X,Y}g_i\ge 1$. It is enough to show that for each factor $g_i(X,Y)$ there are $\ll q^{n/2}$ pairs $Q_1\not\sim Q_2\in M_n$ with $g_i(Q_1,Q_2)=0$. For a given $i$, if we have $\deg_{X,Y}g_i\ge 2$ this follows from \Cref{lem: hit}. If $\deg_{X,Y}g_i=1$, write $g_i=aX+bY+c$, assume WLOG that $b\neq 0$ and denote $\zeta=-a/b$. By \Cref{lem: zeta} we have $\zeta^d=1$.

We distinguish two cases: if $\zeta\neq 1$ then $a\neq -b$ and the leading coefficient of $g_i(Q_1,Q_2)$ is $a+b\neq 0$ for any $Q_1,Q_2\in M_n$ with $n>\deg c$, and therefore there are no such pairs with $g_i(Q_1,Q_2)=0$. If on the other hand $\zeta=1$ (i.e. $a=-b$) then by \Cref{lem: zeta} we have $c/b\in V_f$ and $g_i(Q_1,Q_2)=aQ_1+bQ_2+c=0$ is equivalent to $Q_1+c/b=Q_2$, so $Q_1\sim Q_2$ and there are no pairs with $g_i(Q_1,Q_2)=0$, $Q_1\not\sim Q_2\in M_n$. This completes the proof. 
\end{proof}

\begin{rem}\label{remark: Sfn} It is readily seen from the proof that if $f$ is special in the sense of \Cref{special defention} then we may put 0 on the RHS of \Cref{prop: fq1=fq2} for $n$ sufficiently large, since in this case $\deg_{X,Y}g_i=1$ for all $i$. \end{rem}

\section{The lower bound}
\label{sec: lower bound}

\begin{lem}
\label{lower bound lemma}
     Let $P \in \A$ be a prime such that 
     \begin{enumerate}
         \item[(a)] $\deg P > \deg f_d + n$.
         \item[(b)] $P\nmid\mathrm{Res}(f,f')$ if $f$ is separable.
         \item[(c)] $P\nmid\mathrm{Res}(f,\frac{\partial f}{\partial T})$ if $f$ is inseparable.
     \end{enumerate} 
     Denote $$B_i  = B_i(P) := \left| \{Q \in M_n : P^i \mid f(Q)\neq 0\}\right|.$$ Then \begin{enumerate}
         \item[(i)] $B_1\le d$.
         \item[(ii)] $B_i\le B_1,\,i\ge 1$.
         \item[(iii)] $B_i=0$ if $i>d$.
     \end{enumerate}
\end{lem}

\begin{proof}
Observe that since $ n + \deg f_d < \deg P^i$, $M_n$ is mapped injectivly into $\A / P^i$ by $Q\mapsto f(Q)\bmod P^i$.

{\bf (i).} By the observation above $B_1\le \rho_f(P)\le\deg(f\bmod P)=d$, so $B_1\le d$.

    {\bf (ii).} If $f$ is separable then \Cref{beta for Hensel} combined with assumption (b) applied iteratively with $\alpha=1,2,\ldots,i-1$ implies that each root of $f$ modulo $P$ has a unique lift to a root modulo $P^i$. Combined with the observation above this implies $B_i\le B_1$. If $f$ is inseparable then by \Cref{Hensel for inseprable} and assumption (c) we have $B_i\le\rho_f(P^i)=0\le B_1$ for $i>1$.
    
    {\bf(iii).} Finally suppose that $i > d$. If there exists $Q \in \A$ such that $P^i \mid f(Q)\neq 0$, then $\deg f(Q) \leq d n + \deg f_d$, while $dn + \deg f_d < i \deg P = \deg P^i$, a contradiction. Hence $B_i=0$ in this case.
\end{proof}

We are ready to prove the first part of \Cref{lower bound}.

\begin{proof}[Proof of \Cref{lower bound}(i)]
From \Cref{lower bound lemma} we get that for all but $O(1)$ primes $P \in \A$ (note that by \Cref{non zero resultant} conditions (b-c) of \Cref{lower bound lemma} are satisfied for all but $O(1)$ primes) such that $\deg P > n + \deg f_d$ we have 
$$\an =\sum_{i\ge 1}\sum_{Q\in M_n\atop{P^i|f(Q)}}1=\sum_{i\ge 1\atop{B_i>0}}B_i\leq d\sum_{i\ge 1\atop{B_i>0}}1 \leq d\bn. $$
Now using \Cref{P_f(n)} and \Cref{small P} we obtain
\begin{equation*}
        \begin{split}
           (d-1)n q^n  \lesssim \deg \frac{P_f(n)}{R_f(n)} = \sum_{\deg P > n + \deg f_d} \an \deg P  \leq d \sum_{\deg P > n + \deg f_d} \bn \deg P \leq d \deg L_f(n), 
        \end{split}
\end{equation*}
hence $\deg L_f(n)\gtrsim\frac{d-1}dnq^n$ as required.
\end{proof}

\begin{rem} There exist polynomials $f$ for which the lower bound given by in \Cref{lower bound}(i) and the upper bound given by \Cref{Upper bound} are equal. Indeed, if $|V_f| = d$ (as in \Cref{ex:1}), then combining \Cref{Upper bound} and \Cref{lower bound}(ii) gives $\deg L_f(n) \sim \frac{d-1}{d}n q^n$ as $n\to\infty$. Hence the lower bound is tight in this case.\end{rem}

To prove the second part of \Cref{lower bound} we first need the following

\begin{lem}
    \label{restricted lower bound lemma}
    Assume $p \nmid d$ or $f_d \nmid f_{d-1}$.
Let $n$ be sufficiently large and $P \in \A$ a prime such that $\deg P > \deg f_d + n$. Then in the notation of \Cref{lower bound lemma} we have  
    $B_1(P) \leq d-1$.
\end{lem}

\begin{proof}
    
    Assume by way of contradiction that there exist distinct $Q_1,..., Q_d$ such that $P \mid f(Q_1), ... , f(Q_d)$. The assumption $\deg P>\deg f_d+n$ implies that $Q_i\bmod P$ are also distinct and therefore 
    \[f(X) \equiv f_d \prod_{j = 1}^d(x - Q_j) \pmod P\]
    Comparing coefficients at $X^{d-1}$ we obtain 
    \[f_{d-1} \equiv  - f_d \sum_{j=1}^d Q_j \pmod P\]
    or
    \[f_{d-1}  + f_d \sum_{j=1}^d Q_j \equiv 0 \pmod P.\]
    Thus $P \mid f_{d-1}  + f_d \sum_{j=1}^d Q_j$ and if $f_{d-1}  + f_d \sum_{j=1}^d Q_j \neq 0$ we would have
    $\deg P \leq n + \deg f_d$ (if $n\ge f_{d-1}$), a contradiction.

    Now we observe that as $Q_1,..., Q_d$ are monic, if  $\deg\sum_{j=1}^{d} Q_j<n$ then $p \mid d$. Therefore for a sufficiently large $n$ if $f_{d-1}  + f_d \sum_{j=1}^d Q_j = 0$ then $p \mid d$ and additionally $f_{d}\sum_{j=1}^{d} Q_j = - f_{d-1} $, hence $ f_d \mid f_{d-1}$. This contradicts our initial assumption and we have reached a contradiction in all cases.
\end{proof}

We are ready to prove the second part of \Cref{lower bound}.

\begin{proof}[Proof of \Cref{lower bound}(ii)]
    From \Cref{restricted lower bound lemma} we get that for $P \in \A$ prime such that $\deg P > n + \deg f_d$ we have  $B_1(P) \leq d-1$. The rest of the proof is now the same as the proof of \Cref{lower bound}(i), with the inequality $\alpha_P(n)\le d\beta_P(n)$ replaced by $\alpha_P(n)\le (d-1)\beta_P(n)$, which has the effect of replacing the constant $\frac{d-1}d$ by 1 in the conclusion.
\end{proof}

\section{The radical of $L_f(n)$}\label{sec: rad}

We assume for simplicity of exposition that $f$ is irreducible, although this assumption can be relaxed to $f$ being squarefree in Theorem \ref{thm:rad} and its proof, which we give in the present section.
Recall that we denote $\ell_f(n)=\mathrm{rad}\, L_f(n)$ and observe that
$$\deg\ell_f(n)=\sum_{P\atop{\beta_P(n)>0}}\deg P.$$
By \Cref{Upper bound} and \Cref{lower bound}(i) we have $nq^n\ll \deg L_f(n)\ll nq^n$ and the asymptotic $\deg \ell_f(n) \sim \deg L_f(n)$ (\Cref{thm:rad}) is equivalent to the following   
\begin{prop}
\label{diff betwen L and ell}
    \[\deg L_f(n) - \deg \ell_f(n) = o(n q^n).\]
\end{prop}

\begin{proof}
    Let us write
    \[\ell_f(n) = \sum_{\deg P > n\atop{\beta_P(n)>0}} \deg P + \sum_{\deg P \leq n\atop{\beta_P(n)>0}} \deg P,\]
    \[L_f(n) = \sum_{\deg P > n} \bn\cdot \deg P + \sum_{\deg P \leq n} \bn\cdot \deg P.\]
    Now using (\ref{eq: bound small primes}) we obtain $$\sum_{\deg P \leq n\atop{\beta_P(n)>0}}  \deg P \leq \sum_{P \leq n + \deg f_d}  \bn\cdot\deg P = O(q^n)$$ and therefore
    \[\ell_f(n) = \sum_{\deg P > n\atop{\beta_P(n)>0}} \deg P + O(q^n),\]
    \[L_f(n) = \sum_{\deg P > n\atop{\beta_P(n)>0}} \beta_P(n)\deg P + O(q^n).\]
    Denote $$S:=\{Q \in M_n:  \exists P \text{ prime, } \deg P > n,\, P^2 \mid  f(Q)\}.$$ 
    We have
    \begin{equation*}
        \begin{split}
             0\le \deg \LL - \deg \ell_f(n) &= \sum_{\deg P > n} \bn \deg P -  \sum_{\deg P > n\atop{\beta_P(n)>0}} \deg P + O(q^n)\\
            & \leq \sum_{\deg P > n\atop{\bn \geq 2}} \bn \deg P  +O(q^n)\\
            & \leq \sum_{Q \in S} \deg f(Q) + O(q^n)\\
            & \ll |S| n + q^n.
        \end{split}
    \end{equation*}
    It remains to prove that $|S| = o(q^n)$ but by \cite[Lemma 7.1]{Poo03}, the set 
    $$W=\{\deg Q \leq n:  \exists P \text{ prime, } \deg P > n/2,\, P^2 \mid  f(Q)\}$$ is of size $o(q^n)$ (this is the key ingredient of the proof; see also \cite[Proposition 3.3]{Lan15} and \cite[(4.7)]{carmon2021square} for refinements of this statement). Since $S \subset W$ we have $|S| = o(q^n)$ and we are done.
\end{proof}

As noted above, this concludes the proof of \Cref{ell = L}.

\begin{rem}\label{rem: ABC} Assuming the ABC conjecture, one can prove that $\log \ell_f(N)\sim\log L_f(N)$ over $\Z$. The argument is essentially the same as in the proof of Theorem \ref{thm:rad} above, using the fact that under ABC one has
$$\big|\{n\le N:\,\exists p \mbox{ prime},\,p>N,\,p^2|f(n)\}\big|=o(N),$$
which follows from \cite[Theorem 8]{Gra98}.\end{rem}

\section{Asymptotics of $L_f(n)$ for special polynomials}\label{sec: special asym}

In the present section we assume that $f=\sum_{i=0}^df_iX^i\in\F_q[T][X],\,\deg f=d\ge 2$ is irreducible and special in the sense of \Cref{special defention}. We will show that $\deg L_f(n)\sim c_fnq^n$ as $n\to\infty$, thus proving \Cref{special = conjecture}.
As in \Cref{sec: heuristic} we consider the equivalence relation on $M_n$ given by $Q_1\sim Q_2\iff Q_1-Q_2\in V_f$ and the quantity $S_f(n)$ defined by (\ref{eq: def Sf}).

By the definition of a special polynomial we may write 
\begin{equation}\label{eq: factor fxfy} f(X)-f(Y)=\prod_{i=1}^d(a_iX+b_iY+c_i),\quad a_i,b_i,c_i\in\F_q[T].\end{equation} Comparing degrees in $X$ and $Y$ shows that $a_i,b_i\neq 0$. Comparing coefficients at $X^d$ and $Y^d$ in (\ref{eq: factor fxfy}) we see that $\deg a_i,\deg b_i\le\deg f_d$. Therefore if $Q_1,Q_2\in M_n$ with $n> \deg c_i$, we have \begin{equation}\label{eq:abcdeg}\deg(a_iQ_1+b_iQ_2+c_i)\le n+\deg f_d.\end{equation} 

Let $Q_1,Q_2\in M_n$ with $Q_1\not\sim Q_2$ and $P$ a prime with $\deg P>n+\deg f_d$ such that $P\,|\,f(Q_1),f(Q_2)$. Then in particular $P\,|\,f(Q_1)-f(Q_2)$ and by (\ref{eq: factor fxfy}) we have $P\,|\,a_iQ_1+b_iQ_2+c_i\neq 0$ for some $1\le i\le d$. By (\ref{eq:abcdeg}) and the condition $\deg P>n+\deg f_d$ we must have $a_iQ_1+b_iQ_2+c_i=0$ and therefore by (\ref{eq: factor fxfy}) we have $f(Q_1)=f(Q_2)$.
By \Cref{prop: fq1=fq2} combined with \Cref{remark: Sfn} this is impossible for $n$ sufficiently large. Hence we have $S_f(n)=0$ for $n$ sufficiently large and by \Cref{Eqiv to conjecture} it follows that \Cref{the conjecture} holds for $f$, concluding the proof of \Cref{special = conjecture}.

\section{Classification of special polynomials}\label{sec: class special}

In the present section we classify the special polynomials (\Cref{special defention}) over an arbitrary unique factorization domain (UFD) $R$, establishing \Cref{clasifaction of spesial polinomials}. We denote by $p$ the characteristic of $R$ (as $R$ is a domain, $p$ is zero or a prime) and by $K$ its field of fractions. Also if $p > 0$ we will denote by $\mathbb{F}_p$ its prime subfield. Note that as $R$ is a UFD, so are the polynomial rings $R[X], R[X, Y]$. For a polynomial $g\in R[X,Y]$ we denote by $\deg g$ its total degree. Special polynomials $f\in R[X]$ were defined in \Cref{special defention}.

\begin{lem}
\label{lemma non zero coefficients in special polynomials}
     Let $f\in R[X]$ be a special polynomial of degree $d$. Then we may write
     \begin{equation}\label{eq:fxfy}f(X)-f(Y)=\prod_{i=1}^d (a_i X + b_i Y + c_i),\quad a_i\neq 0,b_i\neq 0,c_i \in  R.\end{equation}
\end{lem}

\begin{proof}
    This is immediate from the definition, except for the condition $a_i,b_i\neq 0$ which follows by comparing degrees in $X$ and $Y$ in (\ref{eq:fxfy}).
\end{proof}

\begin{lem}
\label{lemma for coefficients of special polynomials }
    Let $f=\sum_{i=0}^d f_iX^i\in R[X]$ be a special polynomial of degree $d$.
    \begin{enumerate}\item[(i)]
    If $p=0$ then \[f(X) - f(Y) = f_d\prod_{j=1}^{d} \left(X - \zeta^j Y - b^{(j)}\right),\]
    where $\zeta\in R$ is a primitive $d$-th root of unity and $b^{(j)}\in K$.
    \item[(ii)] If $p>0$ write $d=p^lm$ with $(m,p)=1$. Then \[f(X) - f(Y) = f_d\prod_{i=1}^{p^l} \prod_{j=1}^{m} \left(X - \zeta^j Y - b_i^{(j)}\right),\] where $\zeta\in R$ is a primitive $m$-th root of unity and $b_i^{(j)}\in K$.
    \end{enumerate}
\end{lem}

\begin{proof}
    For brevity we only treat the case $p>0$, the case $p=0$ being similar to the case $p>0,\,l=0$. Write $f(X) - f(Y) = \prod_{i=1}^d (a_i X +b_i Y + c_i)$ as in \Cref{lemma non zero coefficients in special polynomials}. 
    Note that $$\prod_{i=1}^d (a_i X + b_i Y + c_i) = \prod_{i=1}^d (a_i X + b_i Y) + A[X, Y],$$ where $A[X,Y] \in R[X,Y],\deg A < d$. Since $\prod_{i=1}^d (a_i X + b_i Y)$  is homogeneous of degree $d$, we have
    \[f_d (X^d - Y^d) = \prod_{i=1}^d (a_i X + b_i Y)=\prod_{i=1}^d a_i\prod_{i=1}^d(X-\mu_i Y),\]
    where $\mu_i=-b_i/a_i\in K$. Plugging in $Y = 1$ and noting that $\prod_{i=1}^d a_i=f_d$ (by comparing coefficients at $X^d$) we obtain
    \[ X^d - 1 = \prod_{i=1}^d (X - \mu_i).\]
    In particular we see that all the $d$-th roots of unity in the algebraic closure of $K$ lie in $K$, and since $R$ is integrally closed (because it is a UFD) they lie in $R$. Since $p=\mathrm{char}(K)>0$ these are actually the $m$-th roots of unity, where $d=mp^l,\,(m,p)=1$.
\end{proof}



We are ready to prove the forward direction of \Cref{clasifaction of spesial polinomials}.

\begin{prop}\label{prop: classification forward} Let $f\in R[X]$ be special. Then it has the form stated in \Cref{clasifaction of spesial polinomials}.\end{prop}

\begin{proof}
    For brevity we only treat the case $p>0$, the case $p=0$ being similar to the case $p>0,\,l=0$ we treat below (using part (i) of Lemma \ref{lemma for coefficients of special polynomials } instead of part (ii) and the fact that $V_f=\{0\}$ (see \Cref{rem: V_f char 0}). Write $d = \deg f = p^lm, (m,p) = 1$. Denote by $\zeta$ a primitive $m$-th root of unity. From \Cref{lemma for coefficients of special polynomials }(ii) we have $\zeta\in R$ and
    \begin{equation}\label{eq:fxfy prod} f(X) - f(Y) = f_d\prod_{i=1}^{p^l} \prod_{j=1}^{m} \left(X - \zeta^j Y - b_i^{(j)}\right),\quad b_i^{(j)}\in K.\end{equation}

    Next consider the shifted polynomial $g(X)=f(X+w)\in K[X]$ with
    $$w=\left\{\begin{array}{ll}\frac{b_1^{(m-1)}}{1-\zeta^{-1}},&m>1,\\0,&m=1.\end{array}\right.$$
    Note that $g$ is also special but over the ring $K$ instead of $R$ (this is immediate from the definition) and also by the choice of $w$ we have $$g(\zeta X)=g(X),$$ because
    \[g(X) - g(\zeta X) = f(X + w) - f(\zeta X +  w) = f_d\prod_{i=1}^{p^l} \prod_{j=1}^{m} \left(X + w - \zeta^j (\zeta X + w)  - b_i^{(j)}\right)\]
    and if $m>1$ then $X+w-\zeta^{m-1}(\zeta X+w)-b_1^{(m-1)}=0$ (using $\zeta^m=1$). If we can show that $g$ has the form asserted by \Cref{clasifaction of spesial polinomials} except the coefficients are in $K$ instead of $R$, then $f(X)=g(X-w)$ would have the form asserted by the theorem. Replacing $f$ with $g$, we see that it is enough to prove the assertion under the additional assumptions that $R=K$ is a field and $f(X)=f(\zeta X)$.
    
    Thus from now on we assume that $R=K$ is a field and $f(X) = f(\zeta X) = f(\zeta^2 X) = ... = f(\zeta^{m-1} X)$. 
    Next let us prove that for all $j$ the multisets $$S_j := \left\{b_{i}^{(j)} : 1 \leq i\leq p^l\right\}$$ are the same for all $1\le j\le m$. To this end let us identify the indices $1\le j\le m$ with their corresponding residues in $\Z/m$ and observe that by (\ref{eq:fxfy prod}) we have
    \begin{multline}\label{eq: prod multisets}\prod_{j\bmod m}\prod_{b\in S_j}\left(X-\zeta^jY-b\right)=f(X)-f(Y)=f(X)-f(\zeta^kY)=\prod_{j\bmod m}\prod_{b\in S_j}\left(X-\zeta^{j+k}Y-b\right)
    \\=\prod_{j\bmod m}\prod_{b\in S_{j-k}}\left(X-\zeta^j Y-b\right)\end{multline}
    for each $k\in\Z/m$. From the first equality in (\ref{eq: prod multisets}) we see that the number of times an element $b$ appears in the multiset $S_j$ equals the multiplicity of the factor $X-\zeta^jY-b$ in the factorization of $f(X)-f(Y)$ (recall that $R[X,Y]$ is a UFD) and from the last equality it also equals the number of times $b$ appears in $S_{j-k}$. Hence the multisets $S_1,\ldots,S_m$ are all equal and we may rewrite (\ref{eq:fxfy prod}) as
    \begin{equation}\label{eq: fxfy prod 2}
        \begin{split}
            f(X) - f(Y) &= f_d\prod_{i=1}^{p^l} \prod_{j\bmod m} (X - \zeta^j Y - b_i)=f_d\prod_{b\in S_1}\prod_{j\bmod m}(X-\zeta^jY-b),
        \end{split}
    \end{equation}
    where $b_i=b_i^{(1)}$.
    
    Let us prove the following properties of the multiset $S_1=\{b_1,\ldots,b_m\}$: 
    \begin{enumerate}
        \item The underlying set of $S_1$ is 
        $V_f:=\{b\in R:f(X+b)=f(X)\}$. It is an $\F_p$-linear subspace of $R$, with $|V_f|=p^v$ for some $0\le v\le l$.
        \item Multiplication by $\zeta$ permutes $S_1$ as a multiset.
        \item $V_f$ is furthermore an $\F_p(\zeta)$-linear subspace of $R$.
        \item All elements of $S_1$ have the same multiplicity $p^{l-v}$.
    \end{enumerate}

For (1) observe that by (\ref{eq: fxfy prod 2}) we have
$$b\in S_1\iff X-Y-b\,\mid\,f(X)-f(Y)\iff f(X+b)=f(X),$$ hence $V_f$ is the underlying set of $S_1$. Arguing as in the proof of \Cref{V_f: trivial when degree is not devisable by p} one sees that $V_f$ is an $\F_p$-linear subspace of $R$. Since $|V_f|\le |S_1|=p^l$ we see that $|V_f|=p^v$ with $0\le v\le l$. Replacing $Y$ with $\zeta Y$ in (\ref{eq: fxfy prod 2}) and using $f(Y)=f(\zeta Y)$ gives (2). The properties (1),(2) imply (3). 
    
    


    
 To prove (4) observe that by (\ref{eq: fxfy prod 2}) the multiplicity of $b\in V_f$ in $S_1$ equals the exponent of the factor $X-Y-b$ in the factorization of $f(X)-f(Y)$, but since $f(Y+b)=f(Y)$ it also equals the exponent of $X-Y$ in the factorization of $f(X)-f(Y)$ and is therefore independent of $b$. We have established properties (1-4).

 Now assuming WLOG that $V_f=\{b_1,\ldots,b_{p^v}\}$ we can rewrite (\ref{eq: fxfy prod 2}) as
    \[f(X) - f(Y) =f_d\prod_{i=1}^{p^v} \prod_{j=1}^{m} (X - \zeta^j Y - b_i))^{p^{l-v}}=f_d\prod_{i=1}^{p^v}\left((X-b_i)^{mp^{l-v}}-b_i^{mp^{l-v}}Y^{mp^{l-v}}\right).\]
    Setting $Y=0$ we see that $f$ has the form asserted in \Cref{clasifaction of spesial polinomials} with $A=0, C=f(0)$. This completes the proof.
\end{proof}

Finally we prove the converse direction in \Cref{prop: classification forward}.

\begin{prop}\label{prop: classification converse}
\begin{enumerate}\item[(i)] Assume that $p=0$, that $K$ contains a primitive $d$-th root of unity $\zeta$ and $$f=f_d(X+A)^d+C\in R[X]$$ with $A,C\in K$. Then $f$ is special.
\item[(ii)] Assume that $p>0$,  $d=p^lm,\,(m,p)=1$ and $K$ contains a primitive $m$-th root of unity. Suppose
$$f(X)=f_d\prod_{i=1}^{p^v}(X-b_i+A)^{mp^{l-v}}+C\in R[X],$$ where $A,C\in K$ and $V=\{b_1,\ldots,b_{p^v}\},\,0\le v\le l$ ($b_i$ distinct) is an $\F_p(\zeta)$-linear subspace of $K$. Then $f$ is special.\end{enumerate}
\end{prop}

\begin{proof} {\bf (i).} Denote by $\zeta$ a primitive $d$-th root of unity. Then we have the factorization
$$f(X)-f(Y)=f_d\prod_{i=0}^{m-1}\left(X+A-\zeta^i(Y+A)\right)$$ and $f$ is special.

{\bf (ii).} Since $R$ is a UFD, by Gauss's lemma $f(X)-f(Y)$ factors into linear polynomials over $R$ iff it does over $K$. Hence we assume WLOG that $R=K$ is a field and $A,C,b_i\in R$. It is then immediate from the definition that $f(X)$ is special iff $\frac 1{f_d}(f(X-A)-C)$ is, hence we assume WLOG that $A=C=0,\,f_d=1$, i.e. $f=g^{mp^{l-v}}$ where $$g=\prod_{i=1}^{p^v}(X-b_i).$$

We now show that the following factorization holds:
\begin{equation}\label{eq: f is special}f(X)-f(Y)=\prod_{i=1}^{p^v}\prod_{j=1}^m(X-\zeta^jY+b_i)^{p^{l-v}}.\end{equation}
Since the Frobenius map is a ring homomorphism this is equivalent to
\begin{equation}\label{eq: g is special}g(X)^m-g(Y)^m=\prod_{i=1}^{p^v}\prod_{j=1}^m(X-\zeta^jY+b_i).\end{equation}
Note that since $V$ is an $\F_p(\zeta)$-linear subspace, the map $x\mapsto\zeta^{-j}(x+b_i)$ permutes $V$ and we have $$g(\zeta^{-j}(X+b_i))=\zeta^{-jp^v}g(X)$$ and therefore $g(\zeta^{-j}(X+b_i))^m=g(X)^m$. It follows that $X-\zeta^jY+b_i\,\mid\,g(X)^m-g(Y)^m$ and therefore the RHS of (\ref{eq: g is special}) divides the LHS. Since both sides of (\ref{eq: g is special}) have the same degree $mp^v$ and are monic in $X$, we must have an equality. This establishes (\ref{eq: g is special}) and therefore (\ref{eq: f is special}), which completes the proof.

\end{proof}

\begin{rem} Proposition \ref{prop: classification forward} and its proof work under the weaker assumption that $R$ is an integrally closed domain (not necessarily a UFD). Proposition \ref{prop: classification converse} does require the UFD assumption, but it can be replaced with $R$ being only integrally closed if one assumes that $f_d=1$, because Gauss's lemma works for monic polynomials over any integrally closed domain.\end{rem}

\bibliography{ref}

\begin{bibdiv}
\begin{biblist}

\bib{bateman2002limit}{article}{
      author={Bateman, Paul},
      author={Kalb, Jeffrey},
      author={Stenger, Allen},
       title={A limit involving least common multiples: 10797},
        date={2002},
     journal={Amer. Math. Monthly},
      volume={109},
      number={4},
       pages={393\ndash 394},
}

\bib{BaEn21}{article}{
      author={Bary-Soroker, Lior},
      author={Entin, Alexei},
       title={Explicit {Hilbert}'s irreducibility theorem in function fields},
        date={2021},
     journal={Contemp. Math.},
      volume={767},
       pages={125\ndash 134},
}

\bib{carmon2021square}{article}{
      author={Carmon, Dan},
       title={On square-free values of large polynomials over the rational function field},
organization={Cambridge University Press},
        date={2021},
     journal={Math. Proc. Cambridge Philos. Soc.},
      volume={170},
      number={2},
       pages={247\ndash 263},
}

\bib{cilleruelo2011least}{article}{
      author={Cilleruelo, Javier},
       title={The least common multiple of a quadratic sequence},
        date={2011},
     journal={Compos. Math.},
      volume={147},
      number={4},
       pages={1129\ndash 1150},
}

\bib{Fried2023}{book}{
      author={Fried, Michael~D.},
      author={Jarden, Moshe},
       title={Field arithmetic},
     edition={3},
      series={Ser. Modern Surv. Math.},
   publisher={Springer-Verlag},
        date={2008},
      volume={11},
}

\bib{Gra98}{article}{
      author={Granville, A.},
       title={{ABC} allows us to count squarefrees},
        date={1998},
     journal={Int. Math. Res. Not.},
      volume={1998},
      number={19},
       pages={991\ndash 1009},
}

\bib{Lang2002}{book}{
      author={Lang, Serge},
       title={Algebra},
     edition={3},
      series={Grad. Texts in Math.},
   publisher={Springer New York},
        date={2002},
      volume={211},
}

\bib{Lan15}{article}{
      author={Lando, Guy},
       title={Square-free values of polynomials evaluated at primes over a function field},
        date={2015},
     journal={Quart. J. Math.},
      volume={66},
       pages={905–924},
}

\bib{leumi2021lcm}{article}{
      author={Leumi, Etai},
       title={The {LCM} problem for function fields},
        date={2021},
     journal={M.Sc. thesis at Tel Aviv University},
}

\bib{maynard2019lower}{article}{
      author={Maynard, James},
      author={Rudnick, Zeev},
       title={A lower bound on the least common multiple of polynomial sequences},
        date={2021},
     journal={Riv. Mat. Univ. Parma.},
      volume={12},
}

\bib{TrygveNagel1921}{article}{
      author={Nagel, Trygve},
       title={Généralisation d'un théorème de {Tchebycheff}},
        date={1921},
     journal={Math. Pures Appl.},
      volume={4},
       pages={343\ndash 356},
}

\bib{Poo03}{article}{
      author={Poonen, Bjorn},
       title={Squarefree values of multivariable polynomials},
        date={2003},
     journal={Duke Math. J.},
      volume={118},
       pages={353–373},
}

\bib{rosen2002number}{book}{
      author={Rosen, Michael},
       title={Number theory in function fields},
   publisher={Springer Science \& Business Media},
        date={2002},
      volume={210},
}

\bib{rotman2012introduction}{book}{
      author={Rotman, Joseph~J.},
       title={An introduction to the theory of groups},
   publisher={Springer Science \& Business Media},
        date={2012},
      volume={148},
}

\bib{rudnick2020cilleruelo}{article}{
      author={Rudnick, Ze{\'e}v},
      author={Zehavi, Sa'ar},
       title={On {Cilleruelo’s} conjecture for the least common multiple of polynomial sequences},
        date={2020},
     journal={Rev. Mat. Iberoam.},
      volume={37},
      number={4},
       pages={1441\ndash 1458},
}

\bib{sah2020improved}{article}{
      author={Sah, Ashwin},
       title={An improved bound on the least common multiple of polynomial sequences},
        date={2020},
     journal={Th{\'e}or. Nombres Bordeaux},
      volume={32},
      number={3},
       pages={891\ndash 899},
}

\end{biblist}
\end{bibdiv}
\bibliographystyle{amsrefs}

\end{document}